\documentclass[11pt]{article}
\usepackage{arxiv}
\usepackage[utf8]{inputenc}
\usepackage{algorithm}
\usepackage{algpseudocode}
\usepackage{pifont}
\algblockdefx{MRepeat}{EndRepeat}{\textbf{Repeat}}{}
\algnotext{EndRepeat}

\usepackage{amsmath}
\usepackage{amsthm}
\usepackage{amsfonts}
\usepackage{amssymb}
\usepackage{dsfont}
\usepackage{bm}
\usepackage{blkarray}

\usepackage{graphicx}
\usepackage{subcaption}
\usepackage{array,multirow} %% tables

\usepackage{hyperref}
\usepackage[noabbrev,nameinlink]{cleveref}

\theoremstyle{definition}
\newtheorem{theorem}{Theorem}[section]

\newtheorem{proposition}[theorem]{Proposition}
\newtheorem{lemma}[theorem]{Lemma}

\DeclareMathOperator{\nnz}{nnz}
\DeclareMathOperator{\piv}{piv}
\DeclareMathOperator{\img}{im}

\newcommand{\II}{\mathbb{I}}
\newcommand{\Reduce}{\textsc{Reduce}}
\newcommand{\MakeUpperTriangular}{\textsc{MakeUpperTriangular}}

\title{Accelerating Iterated Persistent Homology Computations with Warm Starts}
\author{
 Yuan Luo \thanks{Two authors contributed equally to this work.} \thanks{Corresponding author}\\
	Department of Mathematics\\
	University of California, Davis\\
	Davis, CA 95616\\
	\texttt{luoyuan9809@gmail.com}\\
\And

Bradley J. Nelson\footnotemark[1] \footnotemark[2] \\
	Department of Statistics\\
	University of Chicago\\
	Chicago, IL 60637\\
	\texttt{bradnelson@uchicago.edu} \\
}

\begin{document}
\maketitle

\begin{abstract}
Persistent homology is a topological feature used in a variety of applications such as generating features for data analysis and penalizing optimization problems.  We develop an approach to accelerate persistent homology computations performed on many similar filtered topological spaces which is based on updating associated matrix factorizations. Our approach improves the update scheme of Cohen-Steiner, Edelsbrunner, and Morozov for permutations by additionally handling addition and deletion of cells in a filtered topological space and by processing changes in a single batch.  We show that the complexity of our scheme scales with the number of elementary changes to the filtration which as a result is often less expensive than the full persistent homology computation.   
Finally, we perform computational experiments demonstrating practical speedups in several situations including feature generation and optimization guided by persistent homology.
\end{abstract}


\section{Introduction}\label{sec:introduction}

Persistent homology is an algebraic invariant of filtered topological spaces commonly used in topological data analysis and in other areas of applied and computational topology.  At its core, persistent homology is typically computed using factorizations of the boundary matrices obtained from applying the chain functor (with field coefficients) to a finite cell complex \cite{ZCComputingPH2005}.  A variety of improvements and optimizations to this algorithm have been developed \cite{desilvaDualitiesPersistentCo2011, chenPersistentHomologyComputation2011,mischaikowMorseTheoryFiltrations2013,bauerClearCompressComputing2014,GUDHI15,otterRoadmapComputationPersistent2017} along with efficient implementations \cite{GUDHI15, Ripser19, Eirene16} which have allowed for the computation of persistent homology of increasingly large filtrations.  However, a variety of problems require not just the computation of persistent homology of a single large filtration but of many related filtrations - examples include feature generation for data in machine learning tasks \cite{carlssonTopologyData2009,giustiTwoCompanyThree2016,cangIntegrationElementSpecific2018,gaoRepositioning8565Existing2020} as well as in continuous optimization problems with persistent homology included in the objective \cite{chenTopologicalRegularizerClassifiers2018,rickardCNN2019, leygonieFrameworkDifferentialCalculus2019,topologyLayerMachine2020, carrierePersLayNeuralNetwork2020,kimEfficientTopologicalLayer2020}. In this work, we develop an update scheme for computing persistent homology which updates the computation for a related problem with a warm-start and this scheme can be used efficiently in applications which require iterated computations.

\paragraph{Background on Persistent Homology} We provide a brief introduction to the necessary building blocks from algebraic topology to describe our algorithms.  For a more complete introduction to computational topology and persistent homology, we refer to \cite{edelsbrunnerHarerBook2010, otterRoadmapComputationPersistent2017}.  A cell complex $\mathcal{X}$ is a collection of contractible cells of varying dimensions in which $q$-dimensional cells are connected to $(q-1)$-dimensional cells with maps on their boundaries.  For simplicity, one may consider simplicial or cubical complexes where these boundary maps are determined combinatorially.  Furthermore, we will only consider finite cell complexes. 
\emph{Homology} (with field coefficients) in dimension $q$ is a functor from a topological category to the category of vector spaces over a field $k$.  The homological dimension $q$ captures information about $q$-dimensional features: $q=0$ encodes connected components, $q=1$ encodes loops, and general $q$ encodes $q$-dimensional voids. 

A filtration, or filtered cell complex, is a sequence of cell complexes related by inclusion
\begin{equation}\label{eq:filtration}
    \mathcal{X}_0 \subseteq \mathcal{X}_1 \subseteq \dots
\end{equation}\emph{Persistent homology} is the application of homology with coefficients in a field $k$ to the filtration in \Cref{eq:filtration}. The result can be considered as a $k[T]$ module \cite{ZCComputingPH2005} where the grading by $T$ contains information about the filtration index, or as a diagram of vector spaces connected by linear maps induced by inclusion known as a type-A quiver representation \cite{ZZtheory2010,Oudot}.  Both representations are characterized up to isomorphism by persistence barcodes which are multisets of pairs $\{(b_i,d_i)\}_{i\ge 0}$ that describe the birth and death of homological features in the filtration.

\paragraph{Computing Persistent Homology} Persistent Homology is computed by first applying the cellular chain functor to cell complexes.  A chain complex $C_\ast(\mathcal{X})$ consists of vector spaces $C_q(\mathcal{X})$, $q=0,1,\dots$ with a basis element for each $q$-dimensional cell, and maps 
\begin{equation}\label{eq:chain_boundary}
    D_q: C_q(\mathcal{X}) \to C_{q-1}(\mathcal{X})
\end{equation}
which map the basis element of a cell to a linear combination of basis elements of cells in its boundary.  The boundary maps have the property $D_{q-1} \circ D_q = 0$, and homology is computed as the quotient vector space
\begin{equation}\label{eq:homology}
    H_q(\mathcal{X}) = \ker D_q / \img D_{q+1}.
\end{equation}
Most algorithms for computing persistent homology are based on computing a factorization of filtered boundary matrix, meaning the rows and columns of $D_q$ are arranged in the order of appearance of cells in the filtration:
\begin{equation}\label{eq:factorization_in_q}
    D_q V_q = R_q,
\end{equation}
where $V_q$ is upper-triangular and $R_q$ is \emph{reduced}, which means that it has unique \emph{low pivots}, i.e. the index of the last non-zero row of each column (if it exists) is unique.  The computation of $R_q$ is implicit in the early work of Edelsbrunner, Lester, and Zomorodian \cite{edelsbrunner2000topological}, an explicit algorithm and analysis for $R_q$ was given by Zomorodian and Carlsson \cite{ZCComputingPH2005}, and then a factorization viewpoint was introduced by Cohen-Steiner, Edelsbrunner, and Morozov \cite{vinesvineyards06} when developing a scheme for updating persistent homology, the starting point for this work. 

 We can obtain the persistence information from the factorization in \Cref{eq:factorization_in_q} for each dimension $q$. Only $R_q$ is needed to read off persistent homology: a $q$-dimensional homology class is born when a cell is added that generates a zero column in $R_q$, and this class dies when the index of the birth cell is the pivot of a column of a cell in $R_{q+1}$ \cite{ZCComputingPH2005}.  It is only necessary to form $V_q$ if one wishes to obtain a representative for the homology class, or, as we shall see, update the decomposition.  
 A variety of optimizations have been developed for efficient computation of persistent homology which are compatible with the formation of $V_q$, particularly the clearing \cite{chenPersistentHomologyComputation2011, desilvaDualitiesPersistentCo2011} and compression \cite{ZCComputingPH2005,bauerClearCompressComputing2014} optimizations which are used by state-of-the-art implementations for computing persistent homology \cite{Eirene16, Ripser19}.  
 Other practical accelerations for persistent homology include the use of discrete Morse theory \cite{mischaikowMorseTheoryFiltrations2013} and efficient data structures \cite{boissonnatSimplexTreeEfficient2014a}.

\paragraph{Motivations} Our work is motivated by several applications in topological data analysis.  First, in exploratory data analysis, one may wish to compute the persistent homology of geometric filtrations (i.e. built using pairwise distances) on point-cloud data.  Sometimes several constructions and metrics may be considered, and there may be large amounts of redundant computation done processing the same region for each choice.  Second, in a variety of data analysis problems persistent homology is computed as a feature for each datum in a data set \cite{Dey2017ImprovedIC, garin2019topological_Classification_of_MNIST, Asaad2017TDA_Image_Tampering, Bae2017BeyondDR, Cang2018IntegrationOE, Qaiser2019FastAA}. Often there is a shared structure which we might expect to exploit.  Finally, recent work using persistent homology in gradient-based optimization \cite{rickardCNN2019, topologyLayerMachine2020, carrierePersLayNeuralNetwork2020, chenTopologicalRegularizerClassifiers2018, kimEfficientTopologicalLayer2020,  leygonieFrameworkDifferentialCalculus2019} creates a situation where a topological space undergoes relatively minor modifications in each gradient step.  We wish to be able to reuse computation to the largest extent possible.

\paragraph{Warm Starts} The idea of simply updating the factorization in \Cref{eq:factorization_in_q} for a series of iterated problems is related to a variety of similar techniques in sparse numerical linear algebra and numerical optimization to update $LU$ factorizations \cite{SNOPT2005, gill1987, reid1982, saundersLUSOLSparseLU}.  Our goal is to re-use a previous computation to the largest extent possible, known as a ``warm start'' to the problem.

\paragraph{Contributions} In this work we provide algorithms to compute persistent homology of one filtration starting from the persistent homology of another by updating the associated matrix factorizations.  We analyze the complexity of this update in terms of how close the two filtrations are to each other, namely in terms of the number of cells added and deleted from the filtration and in terms of how the filtration's order changes.  
This approach generalizes the earlier work of Cohen-Steiner, Edelsbrunner, and Morozov \cite{vinesvineyards06} to include addition and removal of cells from a filtration, and includes an analysis that can be applied to general updates beyond elementary permutations.  We additionally adapt our update schemes to cohomology and show how to incorporate the clearing optimization.
Because we perform all updates in a single batch, our method is better suited for blocked or parallel algorithms than the method described in \cite{vinesvineyards06}.
We provide several examples of how our techniques provide practical speedups for both level set and geometric filtrations, and our implementations are made publicly available at 
\url{https://github.com/YuanL12/TDA_Updating_Persistence}.

\section{Algorithms and Analysis}

\subsection{Matrix Reduction}

\paragraph{Notation} We denote column $j$ of a matrix $A$ as $A[j]$, entry $i$ of a (column) vector $v$ as $v[i]$, and the entry in row $i$ and column $j$ of a matrix $A$ as $A[i,j]$.  We say the (low) pivot of a column vector $v$, denoted $\piv(v)$ is the largest index $i$ such that the entry $v[i]$ is non-zero.

Computation of persistent homology typically uses some form of a matrix (column) reduction algorithm on the boundary matrices $\{D_q\}_{q\ge 0}$.  The earliest form of this algorithm applied to persistent homology was described by Edelsbrunner, Letscher, and Zomorodian \cite{edelsbrunner2000topological}, with restrictions to the finite field with two elements and subsets of $\mathbb{R}^3$.  Zomorodian and Carlsson \cite{ZCComputingPH2005} showed that the algorithm can work for general fields and cell complexes, and connected the algorithm with computing the column-echelon form of $D$.  Cohen-Steiner, Edelsbrunner, and Morozov \cite{vinesvineyards06} phrased the algorithm in terms of a matrix decomposition, and established the notation we use here.  We say a matrix $R$ is \emph{reduced} if every column $R[j]$ is either zero or has a unique pivot among columns of $R$.  The reduction algorithm, \Cref{alg:reduction}, produces a reduced matrix $R$ from an input matrix $D$ using elementary column operations that preserve the grading of columns (meaning column $j'$ can be added to column $j$ only if $j'<j$), which means the transformation can be encoded using an invertible upper triangular matrix $V$
\begin{equation}\label{eq:RU_decomposition}
    DV = R.
\end{equation}
\Cref{eq:RU_decomposition} can be re-written as a factorization $D = RU$, where $U=V^{-1}$, referred to as a $RU$ decomposition of $D$ \cite{vinesvineyards06}. 

\begin{algorithm}
\caption{Reduction Algorithm (pHcol)}
\label{alg:reduction}
\begin{algorithmic}[1]
\Procedure{Reduce}{$A$, $B$}
\State \textbf{Input:} $m \times n$ matrix $A$ and $k \times n$ matrix $B$
\State \textbf{Result:} Reduced matrix $A V'$, and matrix $B V'$, formed in place.

\For{ $j = 1,...,n$}
\While{there exists $j' < j$ such that $i = \piv(A[j]) =\piv(A[j']) > 0$}
\State $\alpha = A[i, j]/A[i, j']$
\State $A[j] = A[j] - \alpha A[j']$
\State $B[j] = B[j] - \alpha B[j']$
\EndWhile
\EndFor

\State \textbf{return} $A$, $B$
\EndProcedure
\end{algorithmic}
\end{algorithm}

\begin{proposition}
\label{prop:reduction_alg_reduces}
At the end of \cref{alg:reduction}, the matrix $A$ is reduced in $O(mn \min\{m,n\}$) field operations.  The modification of matrix $B$ incurs $O(kn\min{m,n})$ field operations.
\end{proposition}
\begin{proof}
We proceed by induction using the for-loop in line 5.  For $j=1$, there are no columns to the left, so $A[j]$ is either 0 or has a new unique non-zero pivot. Now, suppose for all $j' < j$, all columns have been reduced to 0 or to have unique pivots.  Examining column $j$, if $\piv A[j]$ is shared with a column $j'< j$, we eliminate that pivot in line 7.  Because this is the last non-zero in both columns, the pivot value must strictly decrease.  We continue to decrease $\piv A[j]$ in this way until either $A[j] = 0$ or we encounter a new pivot which is not found for any columns $j' < j$.  Thus, at the end of the for-loop, all columns of $A$ will either be 0 or have a unique low pivot.

Because the pivot in a column $j$ is strictly decreasing, the maximum number of iterations of the while-loop in line 7 is at most $m$, Additionally, the number of unique pivots is bounded by $n$ which bounds the number of iterations of the while-loop by $\min{m,n}$.  Each column addition in $A$ takes $O(m)$ field operations, and there are $n$ columns, so the total number of operations is $O(mn\min{m,n})$.  Column additions on $B$ take $O(k)$ field operations, for a total number of $O(kn\min\{m,n\})$ operations in the algorithm.
\end{proof}

\begin{lemma}\label{lem:maintain_factorization}
If $A', B' = \Reduce(A, B)$ and $A, B$ satisfy the property $A = C B$ for some matrix $C$.  Then the outputs satisfy $A' = C B'$. Additionally, if $B$ is upper-triangular then $B'$ is also upper-triangular.
\end{lemma}
\begin{proof}
Both $A$ and $B$ are updated by the same column operations, so are acted on the right by the same matrix $V'$.
Because the only column operations are to add columns with smaller column index to the current column $j$, the action of $\Reduce$ is to multiply both inputs $A$ and $B$ by a unit upper triangular matrix $V'$, so $A' = AV'$ and $B' = BV'$.  Thus, $A' = AV' = CBV' = CB'$.  Finally, because upper-triangular matrices are closed under multiplication if $B$ is upper-triangular, so is $B' = BV'$.  
\end{proof}

In the context of persistent homology, we compute $R_q, V_q = \Reduce(D_q, I)$, where $I$ is the identity matrix of an appropriate size.  This produces the $RU$ decomposition $R_q = D_q V_q$ from the initial identity $D_q = D_q I$ and \Cref{lem:maintain_factorization}.   The adaptation of \Cref{prop:reduction_alg_reduces} yields an asymptotic run time bound of $O(n \max\{m,n\}\min\{m,n\})$, or cubic in the number of cells in a filtration \cite{edelsbrunner2000topological}.  An output-sensitive bound can be obtained in terms of the sum of the squared lengths $|d_i - b_i|^2$ in the persistence barcode \cite{edelsbrunner2000topological}, which can also be applied to our use of the reduction algorithm.

The persistence barcode can be read off from $R_q$ by computing this decomposition for each filtered boundary matrix in a filtered chain complex: a new $q$-dimensional bar is born for each column that is reduced to zero in $R_q$, and this bar dies when the same column index appears as a pivot of a column in $R_{q+1}$ \cite{ZCComputingPH2005}.    Unless the visualization of a representative of each homology class is needed, the matrix $V$ that provides this representative information is often not formed in order to save unnecessary work.  In this case, line 8 of \Cref{alg:reduction} and the input $B$ can be omitted.  

 The decomposition for persistent homology can be obtained in matrix multiplication time with more complex algorithms \cite{ZZmatmultime2011}, but practical implementations use variants of the standard algorithm and sparsity typically makes asymptotic run time bounds pessimistic in practice \cite{otterRoadmapComputationPersistent2017}.

\subsection{Clearing}\label{sec:clearing}
The decomposition $DV = R$ is not generally unique, and so  \Cref{alg:reduction} gives one of many valid decompositions. If $j' < j$ and $\piv R[j'] < \piv R[j]$ or $R[j'] = 0$ we can add $R[j']$ to $R[j]$ altering the pivots in $R$ or the upper triangular structure of $V$.  This observation is tied to the non-uniqueness of homology representatives and different choices of bases for homology.  In the case of clearing optimization, a different decomposition is obtained.  In this section, we adapt the approach of Bauer \cite{Ripser19} to the context of warm starts.

Suppose we have $R_q^0 = D_q V_q^0$, where $V_q^0$ is upper triangular, but $R_q^0$ is not necessarily reduced for $q = 0,1,\dots,r$, where $r$ is the largest dimension boundary matrix that we form.   We first compute $R_r, V_r = \Reduce(R_r^0, V_r^0)$ which is a valid $RU$ decomposition of $D_r$ by \Cref{lem:maintain_factorization}. 

We now proceed by computing $RU$ decompositions $R_q = D_q V_q$ in decreasing order of $q = r-1, r-2,\dots$.  
\begin{proposition}\label{prop:warm_start_clearing}
Suppose $D_q \circ D_{q+1} = 0$, and that we have an $RU$ decomposition $R_{q+1} = D_{q+1} V_{q+1}$, and a decomposition $R_q^0 = D_q V_q^0$ where $V_q^0$ is an invertible upper triangular matrix but $R_q^0$ is not necessarily reduced.  Let $\piv R_{q+1} = \{\piv R_{q+1}[j] >0, j=1,2,\dots \}$. Then
we can obtain a new decomposition $R_q^1 = D_q V_q^1$ where
\begin{equation}
R_q^1[i] = \begin{cases}
0 & i \in \piv R_{q+1}\\
R_q^0[i] & i \notin \piv R_{q+1}
\end{cases}, \qquad V_q^1 = \begin{cases}
R_{q+1}[j] & i = \piv R_{q+1}[j]\\
V_q^0[i] & i \notin \piv R_{q+1}
\end{cases}
\end{equation}
where $V_q^1$ is an invertible upper-triangular matrix.
\end{proposition}
\begin{proof}
If $\piv R_{q+1}[j] = i > 0$, the identity $D_q \circ D_{q+1} = 0$ implies that 
\begin{equation}
D_q (D_{q+1} V_{q+1}) [j] = D_q R_{q+1}[j] = 0
\end{equation}
We can now verify that $R_q^1[i] = D_q V_q^1[i]$ for each column index $i=1,2,\dots$ so we have $R_q^1 = D_q V_q^1$. 

Because $V_q^0$ is invertible, $V_q^0[i]$ must have pivot $i$ otherwise its columns would not be linearly independent.  $V_q^1[i]$ is either identical to $V_q^0[i]$ and so also has pivot $i$, or it is set to $R_{q+1}[j]$ which is selected to have pivot $i$, so $V_q^1[i]$ has pivot $i$ in either case.  Thus, $V_q^1$ is upper triangular, and the columns of $V_q^1$ are linearly independent so it is invertible.
\end{proof}
We can then then $R_q, V_q = \Reduce(R_q^1, V_q^1)$, to obtain a $RU$ decomposition of $D^q$ following \Cref{lem:maintain_factorization}.  The original setting of the clearing optimization uses $R_q^0 = D_q$, and $V_q^0 = I$.  Because the columns $R_q^1[i]$ for which $i$ is a pivot in $R_{q+1}$ can be set to be zero (whence the term clearing) they will not be further reduced, which can save a substantial amount of work, as observed in \cite{desilvaDualitiesPersistentCo2011,chenPersistentHomologyComputation2011}.

\subsection{Permuting Filtration Order}\label{sec:perm_update}

Assuming we have computed decompositions $D_q V_q = R_q$, $q=0,1,2,\dots$, of boundary matrices of a filtered cell complex, we would like to update this decomposition to compute persistent homology of the same cell complex with a different filtration order.  If $D_q'$ is the boundary of this new filtration, then
\begin{equation}
    D_q' = P_{q-1} D_q P_q
\end{equation}
Where $P_{q-1}$ and $P_q$ are permutations of the orderings of $(q-1)$-cells and $q$-cells respectively.  We can then modify the decomposition of $D_q$:
\begin{align}
    P_{q-1} D_q P_q P_q^T V_q &= P_{q-1} R_q\\
    D_q' P_q^T V_q &= P_{q-1} R_q\label{eq:initial_update_perm}
\end{align}
where we use the identity $P_q P_q^T = I$.  There are two obstacles that we must overcome to produce a valid $RU$ decomposition $D_q' V_q' = R_q'$. First, $P_q^T V_q$ is not upper triangular, which we correct with \cref{alg:make_upper_triangular}.  Second $P_{q-1} R_q$ may no longer have unique column pivots, which can be corrected using a second application of \cref{alg:reduction}.  We update the decomposition using \Cref{alg:perm_update}.

\begin{algorithm}
\caption{Correct upper-triangular matrix.}
\label{alg:make_upper_triangular}
\begin{algorithmic}[1]
\Procedure{MakeUpperTriangular}{$A$, $B$}
\State \textbf{Input:} Invertible $n \times n$ matrix $A$, $m \times n$ matrix $B$.
\State \textbf{Result:} Upper-triangular matrix $AT$, matrix $BT$, formed in place.
\State $A, B$ = \Reduce($A$, $B$) \Comment{reduce $A$ using \Cref{alg:reduction}}

\For{ $j = 1,...,n$} \Comment{make $A$ upper-triangular}
	\If {$\piv A[j] \neq j$}
	    \State Find $j'$ such that $j = \piv A[j']$.
		\State Swap column $A[j]$ and $A[j']$
		\State Swap column $B[j]$ and $B[j']$
	\EndIf   
\EndFor

\State \textbf{return} $A, B$
\EndProcedure
\end{algorithmic}
\end{algorithm}

\paragraph{Correctness of \Cref{alg:make_upper_triangular}} We assume that the input $A$ is square and invertible, so its columns are linearly independent.  The matrices $A$ and $B$ are modified in-place by column operations, so we can consider the output $A', B' = \MakeUpperTriangular(A, B)$ as the application of an invertible $n\times n$ matrix $T$ on the right of both inputs: $A' = AT, B' = BT$.  After line 4, the matrix $A$ has been reduced.  Because $A$ is invertible, it can not have any zero columns, and so every column must have a pivot, giving $n$ distinct pivots.  The number of rows of $A$ is the same as the number of columns, so there at most $n$ possible pivots.  Thus, after reduction, every row index of $A$ is used as a pivot.  

Then, in the for loop in line 5, the columns of $A$ are permuted so $\piv A[j] = j$, using the fact that there must exist some column $A[j']$ such that $j = \piv A[j']$.  Afterward, $A$ is upper triangular. \qed

\begin{algorithm}
\caption{Update Decomposition with Permutation}
\label{alg:perm_update}
\begin{algorithmic}[1]
\Procedure{UpdatePermutation}{$R, V, P_r, P_c$} 
\State \textbf{Input:} $m \times n$ matrix $R$, invertible $n \times n$ matrix $V$, $m \times m$ permutation matrix $P_r$, and $n \times n$ permutation matrix $P_c$ 
\State \textbf{Result:} Factorization $D' V' = R'$ where $D'  = P_r D P_c = P_r R V^{-1} P_c$

\State $V = P_c^{T} V$
\State $R=P_r R$
 
\State $V, R = \MakeUpperTriangular(V, R)$ \Comment{\Cref{alg:make_upper_triangular}}

\State $R, V = \Reduce(R, V)$ \Comment{\Cref{alg:reduction}, can use clearing.}
\State \textbf{return} $R, V$
\EndProcedure
\end{algorithmic}
\end{algorithm}

\Cref{alg:perm_update} has a similar purpose to the algorithm of Cohen-Steiner, Edelsbrunner, and Morozov \cite{vinesvineyards06} which breaks up permutations into a sequence of elementary transpositions and applies updates based on one of four cases.  In comparison, \Cref{alg:perm_update} is simpler to state and use of the standard reduction algorithm makes implementation straightforward in a package that already computes persistent homology.  The algorithm of \cite{vinesvineyards06} uses a more specialized data structure to allow row swaps in constant time and avoid sorting row indices in the case of elementary transpositions. 

While \Cref{alg:perm_update} could be adapted to use this optimization, the applications we consider typically permute enough elements of the matrix that there is not a significant disadvantage to using whatever matrix data structure is already used for \Cref{alg:reduction}.

\paragraph{Correctness of \Cref{alg:perm_update}} The operations in lines 4 and 5 are given by \cref{eq:initial_update_perm}, so afterward $D'V = R$.  This decomposition invariant is maintained by applying the same column operations to $V$ and $R$ through the rest of algorithm.  In line 6, we make $V$ upper-triangular.  Finally, in line 7, we reduce $R$, and maintain the upper-triangular structure of $V$ following \Cref{lem:maintain_factorization}, producing a $RU$ decomposition of $D'$.  If we process the permutations to the boundaries in decreasing dimension order, we can use the clearing optimization, following \Cref{prop:warm_start_clearing}. \qed

\paragraph{Complexity} A trivial upper bound for the run time of \Cref{alg:perm_update} comes from the calls to \Cref{alg:reduction}.  However, a tighter bound can be obtained based on how greatly the permutations change the filtration order.  Let $|P|_K$ denote the the Kendall--tau distance between the permutation $P$ and the identity permutation, meaning the number of elementary transpositions required to transform $P$ into the identity permutation \cite{diaconisGroupRepresentationsProbability1988}. 
\begin{theorem}\label{thm:perm_time}
In \Cref{alg:perm_update}, if the final reduction in line 11 takes $N(R,V)$ field operations, then \cref{alg:make_upper_triangular} takes $O(\max\{m,n\}|P_c|_K))$ and 
\Cref{alg:perm_update} performs the update in
\begin{equation}\label{eq:perm_update_bound}
    O((\nnz(V) \log n + \nnz(R) \log m + \max\{m,n\}|P_c|_K) + N(R,V))
\end{equation}
field operations.
\end{theorem}
A proof is given in \Cref{sec:perm_complexity}.  If the final reduction in line 7 is faster than reducing from scratch \Cref{alg:reduction} (this often happens in practice), this means that we expect an advantage to using \Cref{alg:perm_update} when filtration values are not changed too drastically, and since $|P_r|_K =O(m^2)$ and $|P_c|_K = O(n^2)$, the algorithm is also worst-case cubic in the number of cells, comparable to \Cref{alg:reduction} albeit with a worse scaling constant.  Note that we can't expect to do better than this since we use the standard reduction algorithm as a subroutine.

\subsection{General Updates}\label{sec:general_update}
Permutation of filtration order is sufficient for applications such as computing persistent homology of different super-level set filtrations on a fixed complex.  However, we may also wish to insert and delete cells in a filtration. One example where this commonly occurs is in computing persistent homology to an intermediate threshold parameter.  For instance, the persistent homology of the Vietoris--Rips filtration on a finite metric space which is often run to a truncated filtration parameter such as the enclosing radius of the metric space \cite{Eirene16, Ripser19}.  Modifications to the underlying metric space can permute the order of simplicies and also necessitate the addition or deletion of simplices which cross the truncation threshold.

Suppose we have filtrations $F$ and $F'$ and wish to compute the decomposition $D_q' V_q' = R_q'$ for $F'$ starting from the decomposition $D_q V_q = R_q$.  We first compute $I_q$ and $I_{q-1}$, the sets of cell indices which will be deleted from $F$ in dimensions $q$ and $q-1$ to form $F \cap F'$. Next, we compute $I_q'$ and $ I_{q-1}'$, the sets of cell indices which are in $F'- F \cap F'$ in dimensions $q$ and $q-1$ which will be added to the filtration $F'$. Finally, we compute $P_q$ and $P_{q-1}$, the permutations of filtration order on the $q$ and $q-1$ cells that are present in $F\cap F'$.  The key observations for our procedure are that in the context of the matrix decomposition $DV = R$,
\begin{enumerate}
    \item\label{it:delete} Cells at the end of a filtration are trivial to remove without altering the upper-triangular structure of $V$ and the reduced structure of $R$; 
    \item\label{it:insert} Cells can be inserted in arbitrary locations without altering the upper-triangular structure of $V$. 
\end{enumerate}
Observation \ref{it:delete} follows the fact that if a $q$-cell is the final cell in a filtration then its column in $D_q$ will not be used to reduce any other columns in the $RU$-decomposition since it is furthest right. Furthermore, its row in $D_{q+1}$ will be the last row and will be entirely zero since it can not appear as a face in a valid filtration (since it is the final cell and faces must appear before a cell can appear).  In contrast to observation \ref{it:delete}, deleting rows and columns in the middle of the filtration would require updating columns to the right which use the deleted column in their reduction.  In order to exploit this observation, we form permutations $Q_r$ and $Q_c$ which permute the filtration order of $F$ so that cells $F\cap F'$ are in the filtration order of $F'$ (specified by $P_r$ and $P_c$ respectively), and the remaining cells in $F$ are permuted to the end of the filtration. 

Observation \ref{it:insert} is easy to see, since adding columns to the boundary $D_q$ (and thus rows and columns to $V_q$) does not invalidate the upper-triangular structure of $V_q$, although a final pass of \Cref{alg:reduction} is required to ensure $R_q$ and $R_{q+1}$ are reduced.  We incorporate these observations into \Cref{alg:gen_update}, which generalizes \Cref{alg:perm_update}.

\begin{algorithm}
\caption{Update Decomposition with Permutation, Insertion, and Deletion}
\label{alg:gen_update}
\begin{algorithmic}[1]
\Procedure{GeneralUpdate}{$R, V, Q_r, Q_c, k_r, k_c, I_r', I_c', D'_c$}
\State \textbf{Input:} $m \times n$ matrix $R$; $n \times n$ matrix $V$; $m\times m$ permutation matrix $Q_r$; $n\times n$ permutation matrix $Q_c$; $k_r$ and $k_c$: respective number of rows and columns deleted from $D$; $I_r'$ and $I_c'$:  respective indices of rows and column indices to insert to form $D'$; $m - k_r + |I_r'| \times |I_c'|$ matrix $D_c'$ containing columns to be inserted.
\State \textbf{Result:} Factorization $D' V' = R'$ incorporating updates
\State $V = Q_c^T V$
\State $R = Q_r R$
\State $V, R = \MakeUpperTriangular(V, R)$ \Comment{\Cref{alg:make_upper_triangular}}

\State Delete the final $k_r$ rows and $k_c$ columns of $R$ and the final $k_c$ rows and columns of $V$.
\State Insert rows of zeros into $R$ at final locations $I_r'$.
\State Insert columns $D'_c$ into $R$ at final locations $I_c'$.
\State Insert rows and columns specified by $I_c'$ in $V$ which act as the identity.

\State $R, V = \Reduce(R, V)$ \Comment{\Cref{alg:reduction}, can use clearing}
\State \textbf{return} $R, V$
\EndProcedure
\end{algorithmic}
\end{algorithm}

\paragraph{Correctness of \Cref{alg:gen_update}} The key idea of the algorithm is to preserve the decomposition identity $D V= R$ without modifying the boundary matrix directly.

Lines 4 and 5 apply row permutations following \cref{eq:initial_update_perm} to put the rows and columns to be deleted in the final blocks.  Next, in line 6, $V$ is made upper triangular so we can safely apply Observation \ref{it:delete}.
In fact, we obtain a decomposition of the permuted boundary matrix $\tilde{D} = Q_r D Q_c$:
\begin{equation}\label{eq:deletion_permutation}
\tilde{D}V' = \begin{bmatrix}
\tilde{D}_{1,1} & \tilde{D}_{1,2}\\
0 & \tilde{D}_{2,2}
\end{bmatrix}\begin{bmatrix}
V'_{1,1} & V'_{1,2}\\
0 & V'_{2,2}
\end{bmatrix} = \begin{bmatrix}
R'_{1,1} & R'_{1,2}\\
0 & R'_{2,2}
\end{bmatrix} = R'.
\end{equation}
The lower-left block of the matrix $\tilde{D}$ is zero, because we may not delete any faces of cells which stay in $F \cap F'$. The lower left block of $R'$ is zero because those columns are linear combinations of columns in the first block of columns in $\tilde{D}$.  Line 7 then deletes all but the top-left block in each of the matrices in \cref{eq:deletion_permutation}, following observation \ref{it:delete}, which can be explicitly confirmed by verifying that $R'_{1,1} = \tilde{D}_{1,1} V'_{1,1}$ after performing the block matrix-matrix multiplication.  
After row and column deletions, we are left with the above decomposition of the boundary matrix $\tilde{D}_{1,1}$ which corresponds to the boundary matrix of $F\cap F'$ in the filtration order of $F'$. We will still use $V'$ and $R'$ to denote their block $V'_{1,1}$ and $R'_{1,1}$ and $\tilde{D}$ for $\tilde{D}_{1,1}$. 

From line 8 to 10, we consider how to transform $\tilde{D}$ to the desired boundary matrix $D'$ of $F'$ by adding cells in $F'$. Assuming the decomposition is for the boundary in dimension $q$, we now consider insert the new faces of dimension $q-1$ in $F'$. This requires us to insert rows of zeros into $\tilde{D}$, which corresponds to line 8, where we insert the rows of zeros to $R'$. Note that it does not affect the upper-triangular structure of $V'$.  
In line 9, we add columns to $R'$ corresponding to the new cells of dimension $q$ in the filtration order of $F'$.  Because we have not yet modified these columns, in line 10 we insert rows and columns in $V'$ to act as the identity, and we now have the identity $D'V' = R'$, where $D'$ is the $q$-dimensional boundary of $F'$ in filtration order. However, while $V'$ is still upper-triangular, the newly inserted columns of $R'$ may have introduced duplicated pivots.  

Finally, in line 11 we make a call to \cref{alg:reduction} to finish the reduction of $R'$, which can use clearing following \cref{prop:warm_start_clearing}. \qed

\paragraph{Complexity} Again, we are interested in an no-worse-than-cubic bound for \Cref{alg:gen_update}, and defer proof to \Cref{sec:general_complexity}.

\begin{theorem}\label{thm:gen_time}
In
\Cref{alg:gen_update}, if the final reduction in line 11 takes $N(R,V)$ field operations, then \Cref{alg:gen_update} performs the update in
\begin{equation}
O(\nnz(V) (\log n + |I'_c|) + \nnz(R) (\log m + |I'_r|)+ \max\{m,n\}|Q_c|_K +N(R,V))
\end{equation}
field operations, where $|Q_c|_K$ is the Kendall--tau distance between the permutation $Q_c$ and the identity permutation.
\end{theorem}

\subsection{Adaptations to Persistent Cohomology}
Persistent cohomology \cite{desilvaDualitiesPersistentCo2011} is a dual algorithm to compute persistence barcodes which are identical to those computed in persistent homology.  In conjunction with the clearing optimization, persistent cohomology can be very efficient compared to homology on filtrations such as the Vietoris--Rips filtration \cite{Ripser19}.  The starting point is to compute $RU$-decompositions of the anti-transposed boundary (or coboundary) matrices $D^{q} = J D_{q+1}^T J$ where the $J$ operator is an anti-diagonal permutation of appropriate dimension which reverses row order when applied on the left or column order when applied on the right.  As a result, the row and column indices of $D^q$ are put in decreasing filtration order as opposed to increasing filtration order for $D_{q+1}$.

By properties of the transposition, we have $D^{q+1} \circ D^q = 0$, so to employ the clearing optimization, we must now process the matrices in order of increasing $q$, and the adaptation of \Cref{prop:warm_start_clearing} is entirely straightforward.  \Cref{alg:perm_update} can be applied to permutations of the filtration with no modification.  However, \Cref{alg:gen_update} requires modifications to the handling of insertions and deletions.

The first modification is that instead of inserting non-zero columns into the boundary $D_{q+1}$ we insert non-zero rows into the coboundary $D^q$.  However, the matrix we modify is actually $R$ in the decomposition $DV = R$.  Assuming we currently have a decomposition $DV = R$, where $V$ need not be upper triangular and $R$ need not be reduced, and we wish to insert a row $d$ into $D$, we need to form the row $r= dV$ to insert into $R$.

The second modification is that instead of permuting simplices to be deleted to the (2,2) block of \cref{eq:deletion_permutation}, we wish to permute them to the (1,1) block.   Because the rows and columns in $D^q$ are put in reverse filtration order, this will again permute cells for deletion to the end of the filtration -- this makes sense when considering the situation where deletions are primarily occurring when cells cross a truncation threshold, and we wish to minimize $|Q_c|_K$ and $|Q_r|_K$.  Thus, we form $Q_c$ and $Q_r$ to put the cells to be deleted in the (1,1) block of $D$, and put the indices for cells in $F\cap F'$ in reverse filtration order of $F'$.  Note that in \cref{eq:deletion_permutation} that the (2,1)-block of $Q_r D Q_c = \tilde{D}$ is still 0, as any cells which will be deleted can not have cofaces in the remaining filtration, so we can verify that 
\begin{equation}
R'_{2,2} = \tilde{D}_{2,1} V'_{1,2} + \tilde{D}_{2,2} V'_{2,2} = \tilde{D}_{2,2} V'_{2,2}
\end{equation}
and we can safely remove the all but the (2,2)-block of the matrix decomposition.

\begin{algorithm}
\caption{Update Cohomology with Permutation, Insertion, and Deletion}
\label{alg:gen_update_cohomology}
\begin{algorithmic}[1]
\Procedure{GeneralUpdateCohomology}{$R, V, Q_r, Q_c, k_r, k_c, I_r', I_c', D'_c$}
\State \textbf{Input:} $m \times n$ matrix $R$; $n \times n$ matrix $V$; $m\times m$ permutation matrix $Q_r$; $n\times n$ permutation matrix $Q_c$; $k_r$ and $k_c$: respective number of rows and columns deleted from $D$; $I_r'$ and $I_c'$:  respective indices of rows and column indices to insert to form $D'$; $|I_r'| \times n - k_c + |I_c'| \times $ matrix $D_r'$ containing rows to be inserted.
\State \textbf{Result:} Factorization $D' V' = R'$ incorporating updates
\State $V = Q_c^T V$
\State $R \gets Q_r R$
\State $V, R = \MakeUpperTriangular(V, R)$ \Comment{\Cref{alg:make_upper_triangular}}

\State Delete the initial $k_r$ rows and $k_c$ columns of $R$ and the initial $k_c$ rows and columns of $V$.
\State Insert zero columns into $R$ at final locations $I_c'$.
\State Insert rows $D'_r V$ into $R$ at final locations $I_r'$.
\State Insert rows and columns specified by $I_c'$ in $V$ which act as the identity.

\State $R, V = \Reduce(R, V)$ \Comment{\Cref{alg:reduction}, can use clearing}
\State \textbf{return} $R, V$
\EndProcedure
\end{algorithmic}
\end{algorithm}

The correctness of \Cref{alg:gen_update_cohomology} can be obtained from the proof of correctness of \Cref{alg:gen_update} incorporating the above discussion.  The modifications produce a different asymptotic complexity using our assumed data structures.

\paragraph{Complexity} We only need to add an additional cost for rows computation ($D'_r V$ in line 9) to the complexity of  \Cref{alg:gen_update}.

\begin{theorem}\label{thm:cohom_time}
In \Cref{alg:gen_update_cohomology} if the final reduction line 11 takes $O(N(R,V))$ field operations, then \Cref{alg:gen_update_cohomology}  performs the update in
\begin{equation}
O(\nnz(V) (\log n + |I'_c|) + \nnz(R) (\log m + |I'_r|)+ \max\{m,n\}|Q_c|_K +  |I'_r| n^2 + N(R,V))
\end{equation}
field operations, where $|Q_c|_K$ is the Kendall--tau distance between the permutation $Q_c$ and the identity permutation.
\end{theorem}

\section{Computational Complexity}\label{sec:complexity}

In this section, we will assume that matrices are stored as a collection of sparse column vectors as it is standard in implementations for persistent homology \cite{Phat2017} (often implemented as a vector of vectors in C++).  Each sparse column is stored as an array of pairs of nonzero indices and values ordered by increasing index. This format is well-suited for column operations which take the most time in the standard reduction algorithm \cref{alg:reduction}. 
We use $\nnz(A)$ to denote the number of non-zeros in a column or a matrix $A$.

\subsection{Permuting Filtrations}\label{sec:perm_complexity}

Let $D$ be a $m\times n$ matrix with $DV = R$, and $P_r$, $P_c$ be row and column permutations so that $D' = P_r D P_c$.

Applying row permutations to $V$ and $R$ generally will require us to alter each non-zero index and update the sorting of each column. If we consider the application of a row permutation to a $m\times n$ matrix $A$, altering each nonzero index takes $\nnz(A)$ time, and sorting indices of non-zeros in each column using standard algorithms takes
$O(\sum_{j= 1}^n \nnz(A[j]) \log \nnz(A[j]))$ operations. Bounding $\nnz(A[j])$ by $m$ will give $O(\sum_{j=1}^n \nnz(A[j]) \log m) = O(\nnz(A) \log m)$.  Thus, applying the row permutations in \Cref{alg:perm_update} takes $O(\nnz(V) \log n + \nnz(R) \log m)$.

Next, we analyze the complexity of reducing $V$ after the row permutation of \Cref{alg:perm_update}.  

\begin{algorithm}
\caption{Reduction Algorithm (pHrow)}
\label{alg:reduction pHrow}
\begin{algorithmic}[1]
\Procedure{Reduce}{$A$, $B$}
\State \textbf{Input:} $m \times n$ matrix $A$ and $k \times n$ matrix $B$
\State \textbf{Result:} Reduced matrix $A V'$, and matrix $B V'$, formed in place.

\For{ $i = m,...,1$}
\State indices = $[j \mid \piv(A[j]) = i ]$
\State $p$ = indices[0]
\For{$j = $ indices[1,...] }
\State $\alpha = A[i, j]/A[i, p]$
\State $A[j] = A[j] - \alpha A[p]$
\State $B[j] = B[j] - \alpha B[p]$
\EndFor
\EndFor

\State \textbf{return} $A$, $B$
\EndProcedure
\end{algorithmic}
\end{algorithm}

We first give an alternative proof of Identical Output Theorem in \cite{desilvaDualitiesPersistentCo2011}. It tells us that the number of column operations of pHcol \Cref{alg:reduction} and pHrow \Cref{alg:reduction pHrow} are the same, which can also be counted by the number of duplicate pivots when reducing rows from bottom to top. 

\begin{theorem}[Identical Output Theorem \cite{desilvaDualitiesPersistentCo2011}]
Given a $m\times n$ boundary matrix $D$, the outputs $R_r,V_r$ of pHrow and $R_c, V_c$ of pHcol are the same.
\end{theorem}
\begin{proof}

We will look at what pHcol will do to reduce duplicate pivots in each row. 
We prove the theorem by induction. Starting from the last row $m$, because the non-zero structure is the same for pHrow and pHcol, they will perform identical duplicate pivots reduction in the row $m$.
Thus, the two algorithms create the same non-zero structure in row $m-1$ to reduce. Furthermore, reducing pivots in this structure entails all column operations in pHcol that are used to reduce duplicate pivots in row $m-1$, because reducing pivots in row $m-1$ comes either from reducing pivots below $m-1$ or from the original boundary matrix.
Then we repeat the above process until the first row. 
Since all column operations are the same, the two algorithm produces identical outputs. 
\end{proof}

Let $\pi$ be the permutation represented by $P^T$ so that row $\pi(i)$ of $V$ is permuted to row $i$ in the multiplication $P^T V$.
\begin{proposition}\label{prop:V_dup_pivots}
Let $V$ be an $n\times n$ upper-triangular matrix and $P$ be an $n\times n$ permutation matrix. Then, in $P^T V$, the maximum number of pivots that must be eliminated is $|P|_K$.
\end{proposition}
\begin{proof}
Because column operations in \Cref{alg:reduction} only add columns with smaller index to columns with larger index, reducing $P^T V$ will not introduce any non-zeros to the left of $\pi(i)$ in row $i$, and the location of potential non-zeros cannot change.  As a result, any column operations used to eliminate pivots in rows $i'> i$ do not affect the bound of number of duplicate pivots in row $i$. 

We start at the final row $i=n$.  In this case, all pivots to the right of $\pi(n)$ will be eliminated, which is a total of $n - \pi(n)$ potential duplicate pivots.  Now, consider an arbitrary row $i$.  There are potentially $n - \pi(i)$ duplicate pivots to eliminate.  However, any row $i'>i$ with $\pi(i') > \pi(i)$ will have the first non-zero entry in the column $\pi(i')$ that is to the right of $\pi(i)$, so the column $\pi(i')$ is already reduced, and we do not need to perform pivot elimination in row $i$. Thus the total number of eliminations in this row is at most
\begin{equation}
    n - \pi(i) - \sum_{j = i+1}^n \II\{\pi(j) > \pi(i)\},
\end{equation}
where $\II\{\pi(j) > \pi(i)\}$  is the indicator function on $j$.

Summing over all rows, the number of pivots to be eliminated is at most
\begin{align}
    &\sum_{i=1}^n \bigg(n - \pi(i) - \sum_{j=i+1}^n \II\{\pi(j) > \pi(i)\}\bigg)\\
    =&\sum_{i'=0}^{n-1} i'  - \sum_{i=1}^n \sum_{j=i+1}^n \big(1 - \II\{ \pi(j) < \pi(i)\}\big)\label{eq:trans1}\\
    =&\sum_{i'=0}^{n-1} i'  - \sum_{i'=0}^{n-1} i' + \sum_{i=1}^n \sum_{j=i+1}^n \II\{ \pi(j) < \pi(i)\}\label{eq:trans2}\\
    =& \sum_{i=1}^n \sum_{j=i+1}^n \II\{ \pi(j) < \pi(i)\}\label{eq:trans3}\\
    =& |P|_K\label{eq:trans4}
\end{align}
The transformation to \Cref{eq:trans1} comes from taking $i' = n-\pi(i)$ and re-ordering the sum, and using $\II\{\pi(j) > \pi(i)\} = 1 - \II\{ \pi(j) < \pi(i) \}$.  Then, the transformation to \Cref{eq:trans2} uses $\sum_{j=i+1}^n 1 = n - i$ and $\sum_{i=1}^{n} (n-\pi(i)) = \sum_{i=1}^{n} (n-i) $, and the two sums over $i'$ cancel in \Cref{eq:trans3}.  Finally, we note that $\sum_{j=i+1}^n \II\{\pi(j) < \pi(i)\}$ is a sum over the number of elementary transpositions in $P$ which move a row $\pi(j)$ past row $\pi(i)$, and we sum over all transpositions in the permutation (counted at the index $i$ where $\pi(i)$ is moved forward) to give us $|P|_k$.
\end{proof}

\begin{proposition}\label{complexity:MakeUpperTrian}
Let $V$ be an invertible upper-triangular $n\times n$ matrix and $R$ be a $m\times n$ matrix.  Then $\MakeUpperTriangular(P_c^T V, P_r R)$ (\Cref{alg:make_upper_triangular}) takes 
\begin{equation}
    O(\max\{m,n\}|P_c|_K + n)
\end{equation}
 time.
\end{proposition}

\begin{proof}
From \Cref{prop:V_dup_pivots}, we must eliminate $|P_c|_K$ pivots in $P_c^T V$, each of which incurs one column operation each on $V$ and $R$, which takes $O(\max\{m,n\})$ time. It is equivalent to transform $P_c^T V$ to $P_c^T V \tilde{V}$, where $\tilde{V}$ records those column operations.  Then the columns of the matrix $P_c^T V \tilde{V}$ are sorted to be in increasing pivot order, which is accomplished in $O(n)$ time by swapping column pointers.
\end{proof}

\Cref{complexity:MakeUpperTrian} demonstrates that our \Cref{alg:make_upper_triangular} until line 6 is as good as the algorithm of \cite{vinesvineyards06}, which they claim takes a linear time complexity for an elementary transposition. 

We are now left to consider the time complexity of reducing the matrix $R$ in line 7 of \Cref{alg:perm_update}. 
Unfortunately, we are not able to provide a bound better than the cubical time complexity of the standard reduction \Cref{alg:reduction}. 
The difficulty comes from that $R$ experiences three multiplications before reduction:
$P_r^T R \tilde{V} \tilde{P}$, where $\tilde{V}$ comes from the reduction on $P_c^T V$ and $\tilde{P}$ comes from the column permutation. Furthermore, non-zeros of $R$ are not right aligned as the upper-triangular matrix $V$, so adding columns left to right will probably introduce new non-zeros. 

Thus, we conclude that \Cref{alg:perm_update}, excluding line 7, takes $O(\nnz(V) \log n + \nnz(R) \log m + \max\{m,n\}|P_c|_K)$.

\subsection{Addition and Deletion of Cells}\label{sec:general_complexity}

An analysis of \Cref{alg:gen_update} begins similarly.  Again, we apply row permutations for a cost of $O(\nnz(V) \log n + \nnz(R) \log m)$, and it is straightforward to extend the analysis of \Cref{sec:perm_complexity} to the reduction of $V$ in line 6, for a cost of $ O(\max\{m,n\}|Q_c|_K)$. 

Modifying the size of the matrix introduces additional considerations. Because we use a vector of sparse columns to store the matrix, deleting the final $|k_c|$ columns of $V$ and $R$ takes constant time.  Furthermore, because after we have deleted these columns in $V$ there are no non-zeros in the last $|k_r|$ rows, deleting these rows does not affect any entries of the remaining columns so can be done in constant time. Inserting rows of zeros in $R$ and $V$ potentially requires us to modify all non-zero indices in all columns, so inserting $|I'_r|$ rows into $R$ and $I'_c$ into $V$ may cost $O(|I'_r|\nnz(R))$ and $O(|I'_c|\nnz(V))$ operations respectively. Inserting columns can be done in $O(n)$ time by inserting pointers. 
Thus,  \Cref{alg:gen_update} except the final reduction (line 11) takes $O(\nnz(V) (\log n + |I'_c|) + \nnz(R) (\log m + |I'_r|)+ \max\{m,n\}|Q_c|_K)$.

Again, this bound is pessimistic due to sparsity in the matrices $V$ and $R$. In addition, note that if we update from an empty complex, then $|I_c'| = n$ and $O(\nnz(R) |I'_r|)) = O(m n \max(m,n))$, which is the same as the bound of \Cref{alg:reduction}.

\subsection{Cohomology}
The only difference between the complexity of \Cref{alg:gen_update_cohomology} and \Cref{alg:gen_update} is in line 9, where we insert (co)boundary vectors which requires the computation of rows $D'_r V$. For $|I'_r|$ rows, the additional cost is $O( |I'_r| n^2 )$. Thus,  \Cref{alg:gen_update_cohomology} excluding the final reduction (line 11) takes $O(\nnz(V) (\log n + |I'_c|) + \nnz(R) (\log m + |I'_r|)+ \max\{m,n\}|Q_c|_K +  |I'_r| n^2)$.

\subsection{Final Reduction}
In \Cref{thm:perm_time,thm:gen_time,thm:cohom_time}, we add an additional term for the final reduction of $R$ after permutations and insertions have been handled. This final reduction has a worst-case complexity that is identical to the reduction from scratch.  However, the performance of the reduction algorithm is sensitive to the input problem, \cite{Ripser19} and in our experiments we see noticeable speedups.

In contrast, \cite{vinesvineyards06} gives a linear-time bound in terms of the Kendall--tau distance of the permutation from the identity, which is achievable because each elementary transposition causes an update before the next elementary permutation is applied.  Our approach is fundamentally different because handle the permutation in a single batch and the number of non-zeros to be eliminated can grow non-linearly in the number of elementary transpositions. However, while \cite{vinesvineyards06} has a tighter asymptotic bound, our method is amenable to parallelism in the reduction as employed in \cite{bauerClearCompressComputing2014,HYPHA19,zhang2020gpu} and can be performed with data structures implemented in many existing persistent homology libraries.

\section{Examples and Experiments}

Our implementation has been incorporated into the Basic Applied Topology Subprograms (BATS) \cite{factorizationView2019} (\url{https://github.com/CompTop/BATS}) library, which provides a standard data structure for representing a matrix using a collection of columns as well as a variety of options for computing persistent homology including the standard reduction algorithm as well as the clearing \cite{chenPersistentHomologyComputation2011,desilvaDualitiesPersistentCo2011} optimization.  This allows us to compare to several algorithmic options without needing to account for implementation-specific variation.  We also compare to the more highly optimized Gudhi \cite{GUDHI15} and Ripser \cite{Ripser19} packages as well as the commonly used Dionysus library \cite{Dionysus2}.  These packages are all comparable using Python bindings for compiled C++ code (for Ripser, we use the bindings at  \url{https://ripser.scikit-tda.org}).  Our timing results are computed using single processes on machines with Intel Xeon 6248R processors and 16GB of available random access memory.

\subsection{Sub-level Set Filtrations}\label{sec:levelset}

One common filtration used in topological data analysis is obtained through sub-level sets of a function on a topological space. Given a function $f:X\to \mathbb{R}$ we denote a sub-level set as $X_a = f^{-1}((-\infty,a])$, and we consider a filtration via the inclusions $X_a \subseteq X_b$ if $a < b$.  An application of this type of filtration is to single channel images, where an image is considered as a pixel intensity function on a $m\times n$ grid which is extended to a filtration on a cubical complex or a simplicial complex via the Freudenthal triangulation.

We investigate level set persistence using several real and synthetic 2-dimensional image data sets:
\begin{enumerate}
\item \textbf{MNIST} \cite{lecun-mnisthandwrittendigit-2010}:  
A collection of handwritten digit images contains a training set of 60,000 examples, and a test set of 10,000 examples. Each image is $28 \times 28$ pixels. As a default we consider computing persistent homology of each image as an update of a pixel-wise averaged image of the same size.

\item \textbf{Vert-64}: A 3-dimensional rotational angiography scan of a head with an aneurysm used for benchmarking persistent homology in \cite{otterRoadmapComputationPersistent2017}. This data set is a 3-dimensional array of size $512 \times 512 \times 512$, and each voxel is a single real-valued number. We obtained the data set from the repository \cite{volvis}.  In our experiments, we subsample the data to form a $64\times 64\times 64$ image due to the memory overhead of forming the basis $V$.  Our update tests perturbation of the pixels by random noise with mean 0 and variance $0.01$.

\item \textbf{S2D($\sigma$)} (sinusoid-2D): A synthetic $128 \times 128$ image $A$ defined as $A[i,j] = \sin(10\pi i/128) + \cos(10\pi j/128)$.  The updated image adds normally distributed random noise with mean 0 and variance $\sigma$.

\item \textbf{S3D($\sigma$)} (sinusoid-3D): A 3-dimensional analog of the S2D($\sigma$) data on a $32\times 32 \times 32$ cube.  In this case, $A[i,j,k] = \sin(4\pi i/32) + \cos(4\pi j/32) + \sin(4\pi k/32)$.
\end{enumerate}

Persistent homology is often used as a feature generation technique. In the case of images, this requires computation of persistent homology for each image in the data set, which can be  a performance bottleneck in part due to implementation and algorithmic complexity and in part due to lack of hardware acceleration seen in more popular image processing techniques such as convolutions.  We will use the MNIST handwritten digit dataset as an example as it readily admits an interpretation of topological features. For example, an image of the digit ``0'' typically has a robust connected component ($H_0$ bar) and a single robust hole ($H_1$ bar), although smaller features may appear due to variations in pixel intensity (e.g. from variations in how hard a pen was pressed down when writing the digit, or from noise in the digitization process).  

\begin{table}[h]
    \centering
\begin{tabular}{|c||c|c|c|c|c||c|}
\hline
&$d_K$& Extension & Build $D$ & Reduction & Update & Total\\
\hline\hline
 Full
& --
& $5.9 \times 10^{-4}$
& $3.6 \times 10^{-4}$
& $1.1 \times 10^{-3}$
& --
& $2.1 \times 10^{-3}$\\\hline
 Image init.
& 0.19
& $6.3 \times 10^{-4}$
& --
& --
& $8.8 \times 10^{-4}$
& $\mathbf{1.5 \times 10^{-3}}$\\\hline
 Avg. init.
& 0.4
& $6.4 \times 10^{-4}$
& --
& --
& $1.2 \times 10^{-3}$
& $1.8 \times 10^{-3}$\\\hline
 Zero init.
& 0.15
& $5.9 \times 10^{-4}$
& --
& --
& $1.2 \times 10^{-3}$
& $1.8 \times 10^{-3}$\\\hline
 Noise init.
& 0.49
& $6.4 \times 10^{-4}$
& --
& --
& $2.1 \times 10^{-3}$
& $2.7 \times 10^{-3}$\\\hline
\end{tabular}

\caption{Average time to compute persistent homology of 1000 MNIST images by updating different reference images.  The Extension column gives the time to extend the filtration on pixels to a filtration on the complex.  The ``Full'' row performs a new reduction every time, using clearing and without forming $V$.  For the update experiments the decomposition is initialized in different ways. ``Image init.'' uses a randomly selected image; ``Avg. init.'' uses the pixel-wise average image; ``Zero init.'' uses a pixel-wise constant (0) image; ``Noise init.'' uses pixel values drawn i.i.d. from a normal distribution.  The column $d_K$ gives the average normalized Kendall--tau distance to initial filtration.}
\label{tab:mnist_features}
\end{table}

\begin{table}[h]
\centering

\begin{tabular}{|c|c||c|c|c|c|c|c|}
\hline
& & \multicolumn{1}{c|}{MNIST} & \multicolumn{1}{c|}{Vert-64} & \multicolumn{1}{c|}{S2D(0.01)} & \multicolumn{1}{c|}{S2D(0.1)} & \multicolumn{1}{c|}{S3D(0.01)} & \multicolumn{1}{c|}{S3D(0.1)}\\
\hline
\hline
\parbox[t]{2mm}{\multirow{7}{*}{\rotatebox[origin=c]{90}{Freudenthal}}}
& $d_K$
& $1.9 \times 10^{-1}$ %0.19
& --
& $3.3 \times 10^{-3}$ %0.0033
& $2.9 \times 10^{-2}$ %0.029
& $1.4 \times 10^{-2}$ %0.014
& $3.0 \times 10^{-2}$ %0.03
\\\cline{2-8}
& Ripser
& $2.2 \times 10^{-3}$
& --
& $4.0 \times 10^{-2}$
& $\mathbf{4.3 \times 10^{-2}}$
& --
& --
\\\cline{2-8}
& Dionysus
& $3.0 \times 10^{-3}$
& --
& $1.2 \times 10^{-1}$
& $1.3 \times 10^{-1}$
& $1.9 \times 10^{0}$
& $1.9 \times 10^{0}$
\\\cline{2-8}
& Gudhi
& $4.6 \times 10^{-3}$
& --
& $1.9 \times 10^{-1}$
& $2.0 \times 10^{-1}$
& $2.2 \times 10^{0}$
& $2.2 \times 10^{0}$
\\\cline{2-8}
& BATS(c)
& $2.1 \times 10^{-3}$
& --
& $6.6 \times 10^{-2}$
& $6.9 \times 10^{-2}$
& $8.1 \times 10^{-1}$
& $\mathbf{8.7 \times 10^{-1}}$
\\\cline{2-8}
& BATS(u,s)
& $\mathbf{1.5 \times 10^{-3}}$
& --
& $\mathbf{3.9 \times 10^{-2}}$
& $6.1 \times 10^{-2}$
& $\mathbf{6.9 \times 10^{-1}}$
& $9.8 \times 10^{-1}$
\\\cline{2-8}
& BATS(u,c)
& $\mathbf{1.5 \times 10^{-3}}$
& --
& $\mathbf{3.9 \times 10^{-2}}$
& $6.0 \times 10^{-2}$
& $7.0 \times 10^{-1}$
& $1.0 \times 10^{0}$
\\\hline\hline
\parbox[t]{2mm}{\multirow{5}{*}{\rotatebox[origin=c]{90}{Cubical}}}
& $d_K$
& $2 \times 10^{-1}$ %0.2
& $4.5 \times 10^{-2}$ %0.045
& $3.3 \times 10^{-3}$ %0.0033
& $3.0 \times 10^{-2}$ %0.03
& $1.4 \times 10^{-2}$ %0.014
& $3.0 \time 10^{-2}$ %0.03
\\\cline{2-8}
& Gudhi
& $2.7 \times 10^{-3}$
& $2.7 \times 10^{0}$
& $3.0 \times 10^{-2}$
& $3.3 \times 10^{-2}$
& $\mathbf{2.0 \times 10^{-1}}$
& $\mathbf{2.1 \times 10^{-1}}$
\\\cline{2-8}
& BATS(c)
& $2.2 \times 10^{-3}$
& $4.1 \times 10^{0}$
& $6.1 \times 10^{-2}$
& $7.3 \times 10^{-2}$
& $4.4 \times 10^{-1}$
& $4.7 \times 10^{-1}$
\\\cline{2-8}
& BATS(u,s)
& $\mathbf{1.2 \times 10^{-3}}$
& $1.3 \times 10^{1}$
& $\mathbf{2.1 \times 10^{-2}}$
& $\mathbf{3.2 \times 10^{-2}}$
& $2.1 \times 10^{-1}$
& $2.8 \times 10^{-1}$
\\\cline{2-8}
& BATS(u,c)
& $\mathbf{1.2 \times 10^{-3}}$
& $\mathbf{2.1 \times 10^{0}}$
& $\mathbf{2.1 \times 10^{-2}}$
& $\mathbf{3.2 \times 10^{-2}}$
& $2.2 \times 10^{-1}$
& $2.7 \times 10^{-1}$
\\\hline
\end{tabular}

\caption{Average time in seconds to recompute or update persistent homology of super-level set filtrations on synthetic and real data, using either Cubical complexes or the Freudenthal triangulation of a grid. $d_K$ is the normalized Kendall--tau distance between the initial and updated filtrations averaged over experiments.  Ripser \cite{Ripser19}, Dionysus \cite{Dionysus2}, GUDHI \cite{GUDHI15}, and BATS(c) \cite{BATS} recompute persistent homology.  Gudhi and Ripser both use cohomology, and Dionysus and BATS both use homology.  BATS(c) uses clearing and does not form the basis $V$.  BATS(u,s) updates the $RU$ decomposition from the standard reduction algorithm and BATS(u,c) updates the $RU$ decomposition obtained from clearing.  Compared to \Cref{tab:mnist_features}, recompute times include the steps of extension, build $D$, and reduction, and update times include extension and update. Timings are averaged over 1000 updates for MNIST (using Image init. for the updating schemes), 1 update for Vert-64, 100 updates for S2D columns, and 20 updates for S3D columns.  Timings for the Freudenthal triangulation of the Vert-64 data set are excluded due to memory constraints.}
\label{tab:levelset_comp}. 
\end{table}

In \Cref{tab:mnist_features}, we measure the average time to compute persistent homology in dimensions 0 and 1 on 1000 random MNIST images using a 2-dimensional Freudenthal triangulation of the $28 \times 28$ grid for a total of 784 0-simplices, 2241 1-simplices, and 1458 2-simplices. We use a single initial filtration which is updated for each image.  Overall, our update scheme gives almost a 3x speedup compared to a full persistent homology computation.  We observe that initializing with an actual image produces slightly faster updates when compared to an ``average image" produced by averaging each pixel value over the data set or a constant ``zero image".  Note that even initializing with the constant image gives a large speedup.  Because MNIST digits have a constant background using this constant image for initialization is advantageous because much of the factorization can be reused over this constant region.  We also measure the time to update the persistent homology of an ``image'' generated from random pixel values, which still gives a noticeable speedup.  We can use this as a baseline to determine how much of the speedup using a representative image for initialization is due to memory and implementation efficiency and how much is due to the cost of updating persistent homology from a good starting point vs. a bad starting point.

In \Cref{tab:levelset_comp} we measure the time needed to compute persistent homology on a variety of data, either from scratch or using our update scheme.  On all the spaces built on the Freudenthal triangulation of a grid, our update scheme demonstrates a noticeable improvement in run time, and for cubical complexes we outperform Gudhi on smaller and simpler updates, and are slightly outperformed on larger problems and updates.  We also note that Dionyusus has a built-in function for the Freudenthal triangulation of an image whereas Gudhi does not, so the better performance of Gudhi on persistent homology computations is offset by the need to construct the filtration in Python.  We report the results of the clearing optimization in BATS - compression tends to perform slightly worse on these examples.

\subsection{Vietoris--Rips Filtrations}\label{sec:geom_filtration}

Vietoris--Rips filtrations (or simply Rips filtrations) are commonly used in conjunction with persistent homology to create features for finite dimensional metric spaces (point clouds).  Given a metric space $(X, d)$, a Rips complex consists of simplices with a maximum pairwise distance between vertices is less than some threshold $r$:
\begin{equation}
X_r = \{(x_0,\dots,x_k) \mid x_i\in X, d(x_i,x_j) \le r\}.
\end{equation}
A Rips filtration is a filtration of Rips complexes $X_r \subseteq X_s$ if $r \le s$.

The number of simplices in Rips filtrations quickly grows with the size of the data set, and much effort has gone into developing efficient algorithms for computing persistent homology of Rips filtrations.  While it is possible to use an approach such as that done in \Cref{sec:levelset} which is to update every simplex in a filtration, several high-performance packages for Rips computations \cite{Ripser19, Eirene16} stop a filtration at the \emph{enclosing radius} of the metric space, at which point the complex becomes contractible, which can reduce the total number of simplices in the filtration considerably without changing persistent homology.  In order to combine this optimization with our approach, it is necessary to be able to add and remove simplices from filtrations as well as permute their filtration order as in \Cref{alg:gen_update}.

\subsubsection{Updates on different data sets}
We list all data sets used in our experiments below, including synthetic data sets (\ref{ds:sphere})(\ref{ds:Klein3}) and empirical measurements and experiments (\ref{ds:Bunny})(\ref{ds:Dragon})(\ref{ds:H3N2}).

\begin{enumerate}
\item \textbf{Sphere1} and \textbf{Sphere2} \label{ds:sphere}: We first randomly generate two data sets, where each with 200 points on $S^1 \subset \mathbb{R}^2$ and on $S^2 \subset \mathbb{R}^3$, and next add normal noise with standard deviation $0.001$ to them. We update persistence from unnoised spheres.

\item \textbf{Eight}: We randomly generate a figure 8 with 200 points in $ \mathbb{R}^2$ and add normal noise with standard deviation $0.001$(See \Cref{fig:data 8}). We measure the performance our updating scheme after the perturbation of noise scale.

\item \textbf{Klein3} \label{ds:Klein3}: The data set was introduced in \cite{otterRoadmapComputationPersistent2017}, which samples 400 points from the Klein bottle using its “figure-8” immersion in $\mathbb{R}^3$. We randomly re-sample 100 points from it and test our updating scheme on perturbation by normal noise with standard deviation $0.01$. 

\item \textbf{Bunny}\label{ds:Bunny}: The Bunny model comes from the Stanford Computer Graphics Laboratory \cite{Stanford3D}. We use one of its 3D scan picture with size 40256 points in $\mathbb{R}^3$ and (uniform) randomly sample 100 points. Our updating scheme test on perturbation by normal noise with standard deviation $0.01$.
 
\item \textbf{Dragon}\label{ds:Dragon}: It is a 3-dimensional scan of a dragon from the Stanford Dragon graphic model \cite{Stanford3D} and in \cite{otterRoadmapComputationPersistent2017} consists of 1000 and 2000 points sampled uniformly at random. We randomly re-sample 400 points from the 1000 points and test our updating scheme on perturbation by normal noise with standard deviation $0.01$. 

\item \textbf{H3N2}\label{ds:H3N2}: The data set from \cite{otterRoadmapComputationPersistent2017} contains 2722 different genetic sequences of H3N2 influenza, where each sequence is a vector in $\mathbb{R}^{1173}$. There are many genetic metrics used to measure the difference between two genetic sequences, but we will focus on the Euclidean metric and encourage readers to try on different ones. We randomly sample 200 points and and test our updating scheme on perturbation by normal noise with standard deviation $0.01$. 

\end{enumerate}

\begin{figure}[h]
  \centering
  \includegraphics[width=0.4\linewidth]{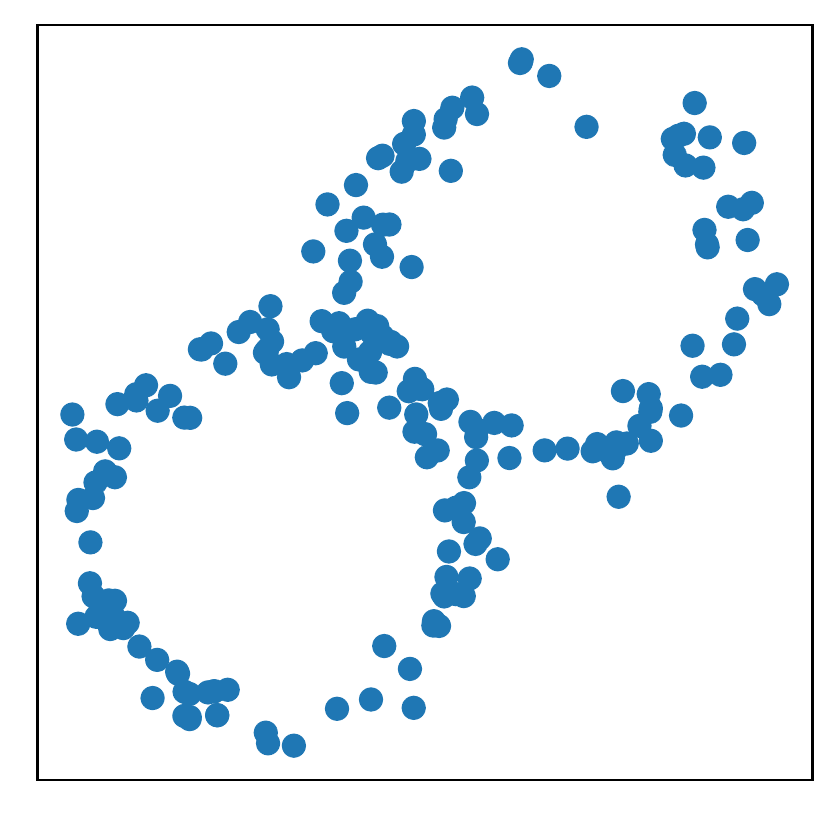}
  \caption{Data set \textbf{Eight}}
  \label{fig:data 8}
\end{figure}

\begin{table}[h]
    \centering
    \begin{tabular}{|c||c|c|c|c|c|c|}
    \hline
    & Sphere1 & Sphere2 & Klein3 & Dragon & Bunny & H3N2\\
    \hline
    max. PH & 1 & 2 & 2 & 1 & 2 & 1 \\
    \hline
    $d_K$ & $4.7 \times 10^{-3}$  & $4.4 \times 10^{-3}$ & $9.6 \times 10^{-4}$ & $2.2 \times 10^{-2}$ & $1.5 \times 10^{-2}$ & $5.8 \times 10^{-3}$\\
    \hline
    Add & $1.7 \times 10^{-3}$ & $9.4 \times 10^{-4}$ & 0 & $2.1 \times 10^{-2}$ & $9.3 \times 10^{-4}$ & $2.8 \times 10^{-4}$ \\
    \hline
    Del & $3.2 \times 10^{-2}$ & $1.2 \times 10^{-2}$ & 0 & $7.1 \times 10^{-3}$ &  $1.7 \times 10^{-2}$& $7.5 \times 10^{-4}$ \\
    \hline
    \hline
BATS(c,enc) & 1.96 & 200.17 & 1.30 & \textbf{0.41} & 2.07 & 0.92 \\
    \hline
BATS(c,b,enc) & 2.08 & 215.45 & 1.46 & 0.48 & 2.18 & 1.04 \\
	\hline
BATS(u,c,enc) & 1.95 & 242.03 & \textbf{0.96} & 0.52 & \textbf{1.95} & 1.83 \\
    \hline
BATS(u,c,full) & \textbf{0.82} & \textbf{62.72} & 2.94 & 1.17 & 3.60 & \textbf{0.83} \\
	\hline
	\hline
Gudhi & 0.60 & 60.59 & 0.87 & 0.26 & 1.54 & 0.59\\
	\hline
Ripser & \textbf{0.12} & \textbf{2.52} & \textbf{0.08} & \textbf{0.02} & \textbf{0.09} & \textbf{0.03}\\
	\hline
    \end{tabular}
    \caption{Average time in seconds to recompute or update persistent homology of Rips filtrations on different data sets. The first four rows are filtration information. `max. PH' is the maximum persistent homology dimension we compute up to. $d_K$ is the normalized Kendall–tau distance between the initial and updated filtrations. `Add' and `Del' are the fractions of the number of cells added to and deleted from the original filtrations divided by the size of original filtrations. The following 4 rows are different algorithms that are implemented in BATS and then the final two rows are GUDHI \cite{GUDHI15} and Ripser \cite{Ripser19}. All algorithms use cohomology to compute persistent homology. BATS(c,enc) uses clearing, does not form the basis $V$ and stops the filtration at the \emph{enclosing radius}. BATS(c,b,enc) uses clearing, forms the basis $V$ and stops at the \emph{enclosing radius}. BATS(u,c,enc) updates the $RU$ decomposition from the cohomology clearing algorithm with two filtration both stop at \emph{enclosing radius}. BATS(u,c,full) updates the $RU$ decomposition with two filtration both stop at infinite radius, where insertion and deletion are not involved in our updating scheme.}
    \label{tab:Rips Experiment}.  
\end{table}

As suggested in \cite{Ripser19}, we found the great efficiency of the cohomology clearing algorithm and so only tested the performance of the cohomology update \Cref{alg:gen_update_cohomology}. In \Cref{tab:Rips Experiment}, each row represents an algorithm in BATS or in another package and each column records the time spent on recomputing or updating persistent homology on a data set. The first four algorithms are all implemented in BATS. BATS(c,enc) uses clearing, does not form the basis $V$ and stops the filtration at the \emph{enclosing radius}. BATS(c,b,enc) uses clearing, forms the basis $V$ and stops at the \emph{enclosing radius}. BATS(u,c,enc) updates the $RU$ decomposition from the cohomology clearing algorithm with two filtration both stop at \emph{enclosing radius}. BATS(u,c,full) updates the $RU$ decomposition with two filtration both stop at infinite radius, where insertion and deletion are not involved in our updating scheme. We make the last two comparisons to see how addition/deletion and permutation will affect the updating performance. The final two rows are GUDHI \cite{GUDHI15} and Ripser \cite{Ripser19} and both of them use cohomology to compute persistent homology. 

The results in \Cref{tab:Rips Experiment} show that our updating algorithm \Cref{alg:gen_update_cohomology} is better than recomputing except for data set \textbf{Dragon}. On \textbf{Sphere2}, our update scheme BATS(u,c,full) demonstrates a noticeable improvement in run time. However, for \textbf{Dragon}, we suspect insertion involved in update is the main bottleneck of its inefficiency, because a new row $r$ in $R$ at line 9 \Cref{alg:gen_update_cohomology} is a product of a row in $D'_r$ and the matrix $V$, which can be a big source of overhead.
For the comparisons between other TDA packages, Ripser demonstrates a large performance advantage over other options, but we note it is specifically optimized for Rips Filtrations.

\subsection{Permutations, Additions, Deletions}

Because the complexity of \Cref{alg:gen_update} and  \Cref{alg:gen_update_cohomology} depends on the size of permutation, insertion and deletion, we explicitly investigate the effect of them by modifying the maximum radius of Rips filtration. We use the data set \textbf{Eight}. To analyze them separately, for permutation, we set the maximum radius to be infinite and change the level of noise; for insertion, we first compute the persistent homology with maximum radius set to be zero and then update persistent on increasing maximum radius to infinity; for deletion, we first compute the persistent homology with maximum radius set to be infinity and then update persistent on decreasing maximum radius to zero. 

In \Cref{fig:Perms Adds Dels}, we see
the linear relation of update time  between the size of insertion and deletion, while update time grow exponentially on the size of permutation measured by the Kendall--tau distance. 

\begin{figure}[h]
\begin{subfigure}{.35\textwidth}
  \centering
  \includegraphics[width=1\linewidth]{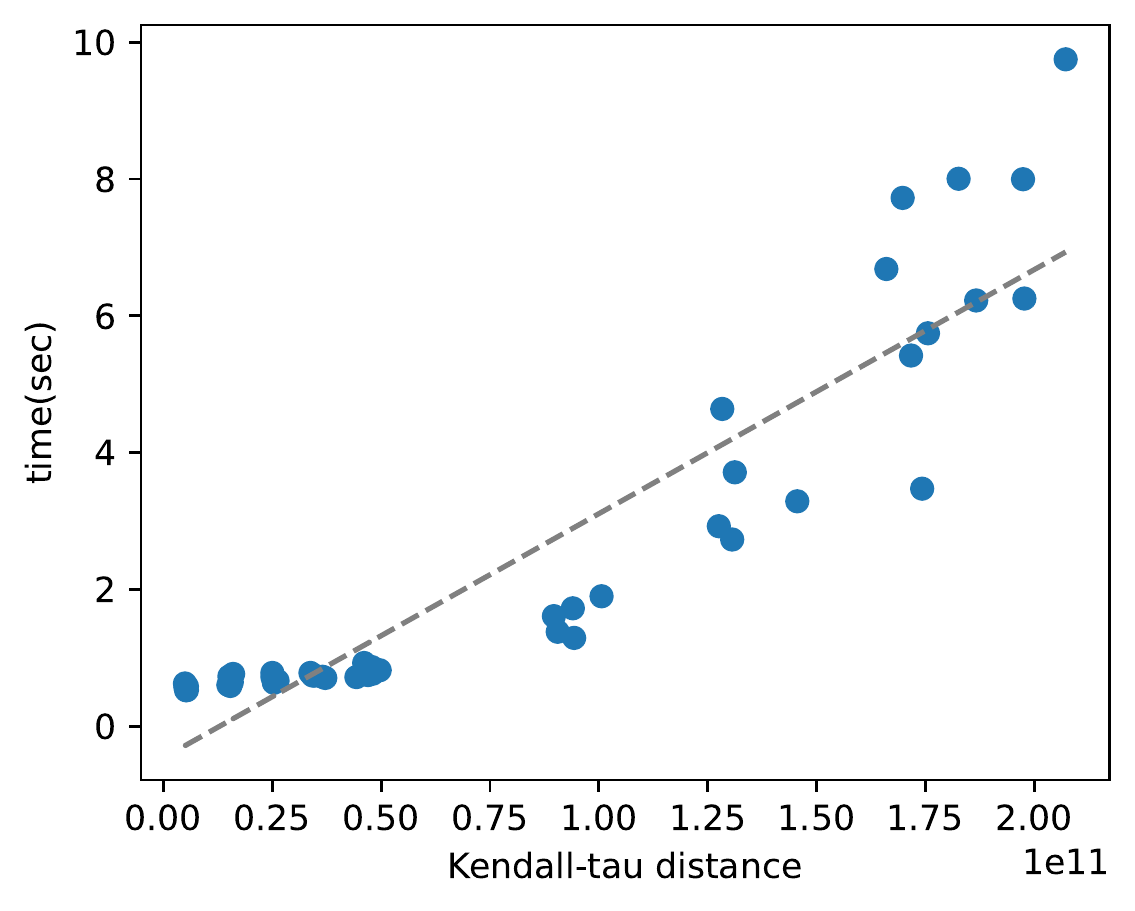}
  \caption{}
  \label{fig:subfig rips perms}
\end{subfigure}%
\begin{subfigure}{.35\textwidth}
  \centering
  \includegraphics[width=1\linewidth]{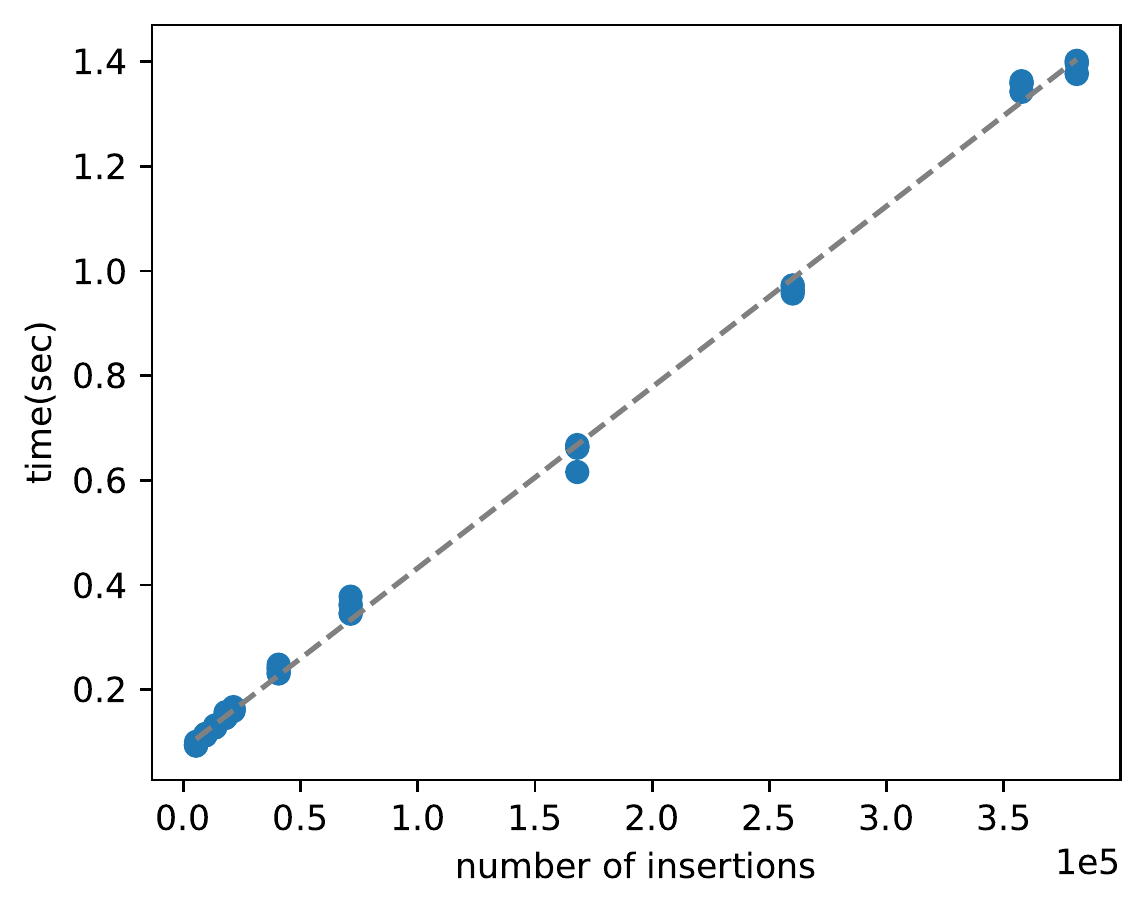}
  \caption{}
  \label{fig:subfig rips inserts}
\end{subfigure}%
\begin{subfigure}{.35\textwidth}
  \centering
  \includegraphics[width=1\linewidth]{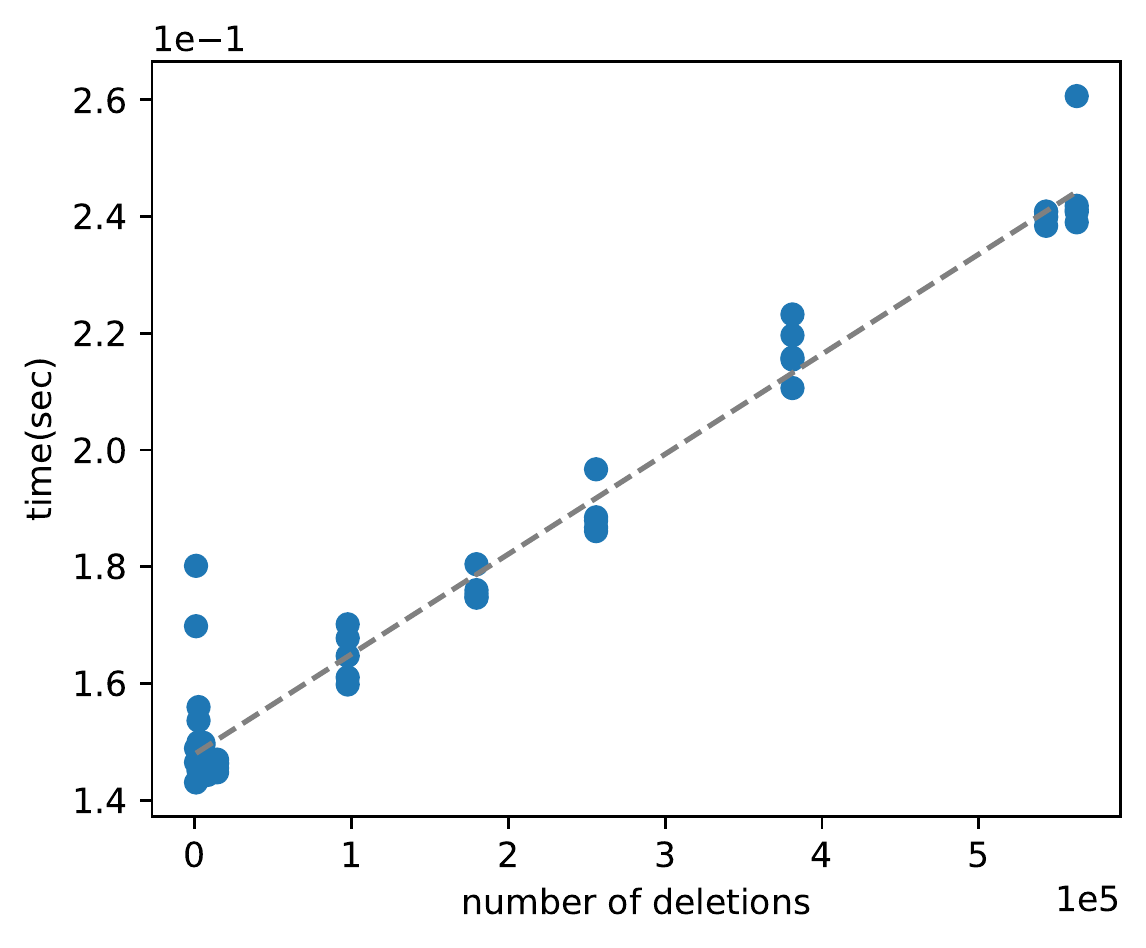}
  \caption{}
  \label{fig:subfig rips dels}
\end{subfigure}%
\caption{
Time of Cohomology Update with clearing on Rips filtration of data set \textbf{Eight} versus (i) Kendall--tau distance of permutations performed in update (\cref{fig:subfig rips perms}), where the maximum distance is 
$2.07 \times 10^{11}$;
(ii) number of insertions (\cref{fig:subfig rips inserts}); (iii) number of deletions (\cref{fig:subfig rips dels}). The dashed grey lines are linear regression lines fit by least squares.}
\label{fig:Perms Adds Dels}
\end{figure}

\subsection{Optimization}
Optimization of a function of persistent homology is another potential application for \Cref{alg:perm_update,alg:gen_update}. Theory of designing  differentiable persistence-based functions and their differential structures are discussed in \cite{Poulenard2018TopologicalFO, Solomon2021AFA, leygonie2021framework}.
We adopt the function defined in  \cite{topologyLayerMachine2020} for the following experiment, 
\begin{equation}
\mathcal{E}\left(p, q, i_{0} ; \mathrm{PD}_{k}\right)=\sum_{i=i_{0}}^{\left|I_{k}\right|}\left|d_{i}-b_{i}\right|^{p}\left(\frac{d_{i}+b_{i}}{2}\right)^{q},
\end{equation}  
where birth $b_i$ and death $d_i$ is a persistence pair of persistence diagram $\mathrm{PD}_{k}$ of Rips filtration at dimension $k$ starting at the $i_0$ longest persistence pair to those of shorter length. We maximize the sum of the lengths of 1-dimensional persistence bars starting from a point cloud sampled uniformly from the unit square.  Explicitly, we use gradient descent to maximize the function $\mathcal{E}\left(2, 0, 1 ; \mathrm{PD}_{1}\right) = \sum_{i = 1}^{I_1} (d_i - b_i)^2 $, where $I_1$ is the number of 1-dimensional persistence pairs $(b_i, d_i)$. As shown in \Cref{fig:rips_opt}, after 100 iterations, points that are originally uniformly generated in the unit square are moved to form more holes. 

In \Cref{tab:Opt rips}, we report on the experiment result. The first two columns records the time of Cohomology clearing algorithm by recomputing and updating, which is much faster than the final two columns by Homology. Despite that, we can observe that for homology algorithm, our updating scheme can achieve a great speedup. We also suspect that the bottleneck of our Cohomology update is insertion of new simplices, because each new row in $R$ requires a vector-matrix multiplication. 

\begin{table}[h]
    \centering
    \begin{tabular}{|c|c|c|c|}
    \hline
    	 Cohomology Clearing & Cohomology Update & Homology Clearing & Homology Update\\
    \hline 
    \textbf{1.30} & 1.73 & 20.78 & 6.96\\
 	\hline
    \end{tabular}
    \caption{Time of 100 iterations to maximize $\max \sum_{i = 1}^{I_1} (d_i - b_i)^2$ from a uniformly generated data set by 4 algorithms implemented in BATS, where $(b_i, d_i)$ are persistence pairs of the Rips filtration. The 4 algorithms are Cohomology with clearing, update with Cohomology (\Cref{alg:gen_update_cohomology}), Homology with clearing and update with Homology(\Cref{alg:gen_update}). The optimization results are shown in \Cref{fig:rips_opt}.}
    \label{tab:Opt rips}
\end{table}

\begin{figure*}[ht!]
   \subfloat[\label{genworkflow}]{%
      \includegraphics[ width=0.4\textwidth]{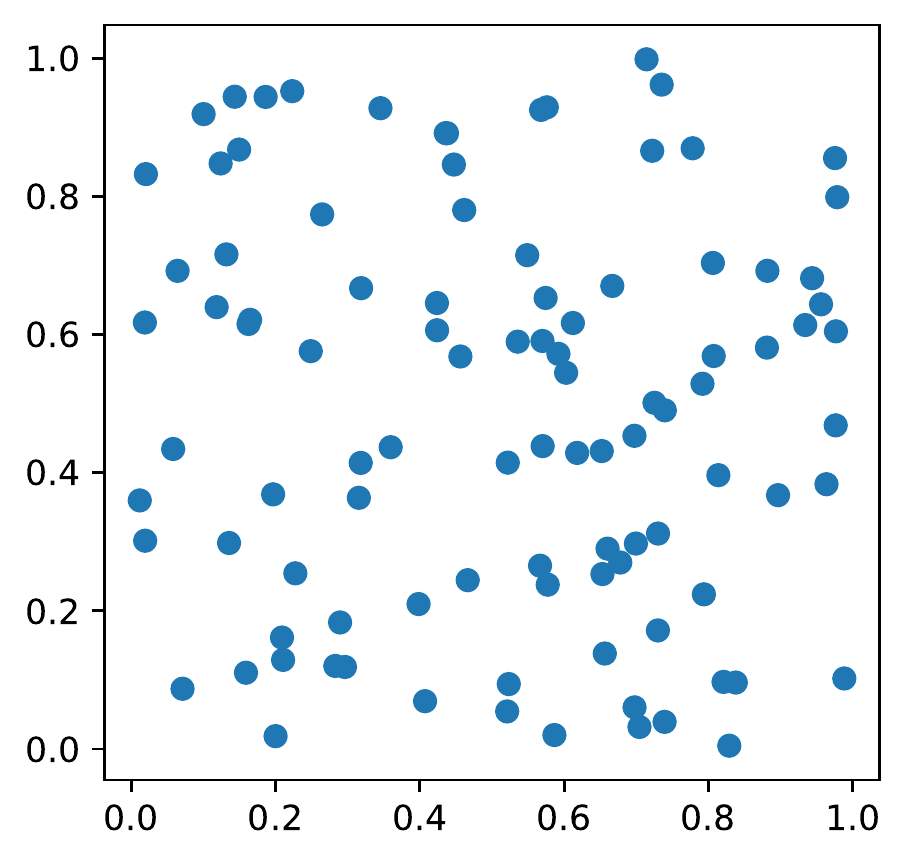}}
\hspace{\fill}
   \subfloat[\label{pyramidprocess} ]{%
      \includegraphics[ width=0.4\textwidth]{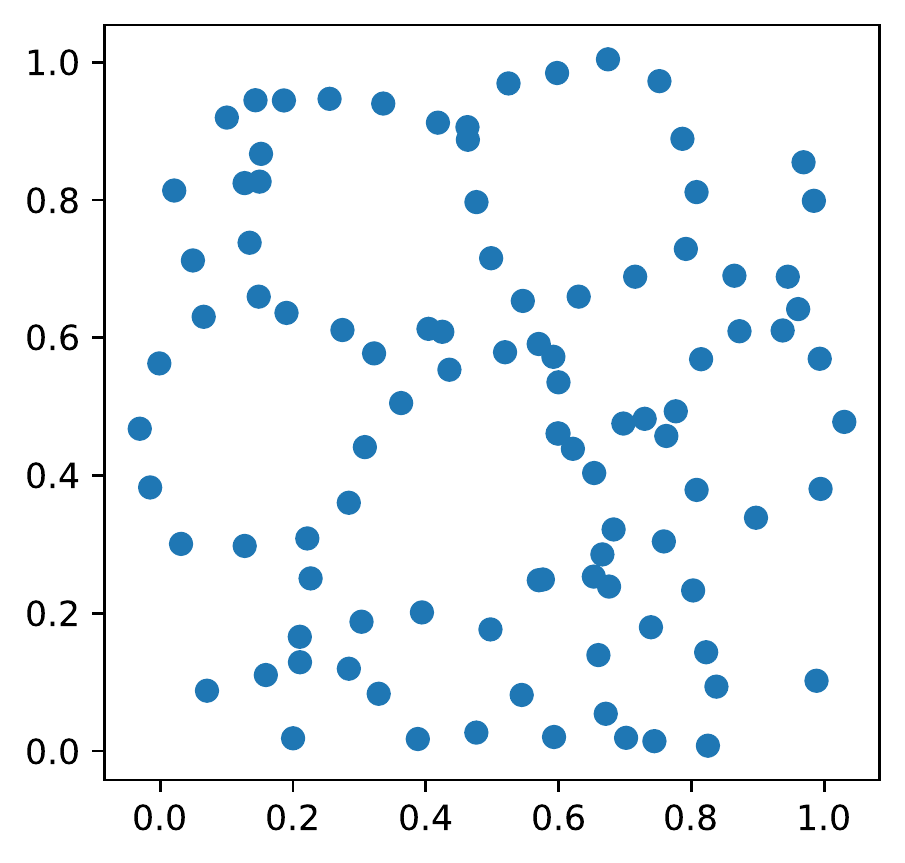}}\\
\caption{Results of 100 iterations to maximize the square length of persistent pairs of Rips filtration at dimension one. We used Cohomology update to transform a uniformly generated dataset Figure (a) to Figure (b).}
\label{fig:rips_opt}
\end{figure*}

\section{Conclusion}

We present two algorithms for updating persistence: one for a fixed-sized filtration and another for a general filtration. The algorithms' asymptotic complexity is shown to be comparable to the standard reduction algorithm for computing persistent homology in the worst case, and we provide tighter bounds based on the details of the updates. Our algorithm demonstrates practical speedups on several experiments, especially where changes to the filtration are limited. 
We implemented our method using the data structures in the Basic Applied Topology Subprograms (BATS) library \cite{BATS}, in order to obtain consistent comparisons with several variations of the reduction algorithm for persistent homology. We see that our update method can provide a speedup of 2-3x in several practical situations. 

While we have demonstrated the utility of our approach in certain situations, there are also some limitations to its use. Some of these are inherent, for instance our approach does not work well when filtrations change too drastically, or when the additional memory requirements of maintaining the matrix $V$ are cost prohibitive. Other limitations may be implementation-specific, for instance we see that Gudhi \cite{GUDHI15} and Ripser \cite{Ripser19} outperform our update scheme on Vietoris--Rips computations.  

Deciding which algorithm to use for computing persistent homology on many similar problems is context-dependent.  For fixed size filtrations, as in level set persistence, using our update scheme appears to be a reasonable choice for smaller perturbations, particularly when maintaining the basis matrix $V$ is desirable.  For geometric filtrations, we recommend using a high-performance package designed for these computations, particularly if the homology basis is not required. In practice, a practitioner may wish to test several options experimentally as run times can be problem dependent.

There are several directions for future investigation which may build on this work.  One direction is to develop methods to limit fill-in in the $RU$ decomposition when performing updates, a problem related to that of finding sparse homology generators \cite{obayashiVolumeOptimalCycleTightest2018a}.  As we have discussed, this appears to be an important consideration in several potential applications of our update schemes such as optimization using level set filtrations.  
There may also be ways to adapt our methods to the context of updating discrete Morse vector fields \cite{mischaikowMorseTheoryFiltrations2013}, which may offer another way to accelerate iterated persistent homology computations.  Finally, because we use the standard reduction algorithm as a black box, we suspect that the application of blocked or parallel methods \cite{bauer2013DistrubutedComputationofPH, zhang2020gpu, Phat2017, HYPHA19} offers a path to improve on the performance seen in our experiments.

\section*{Acknowledgements:} BN was supported by the Defense Advanced Research Projects Agency (DARPA) under Agreement No.
HR00112190040.  We are grateful for compute resources provided by the Research Computing Center (RCC) at the University of Chicago.

\bibliographystyle{acm}      % mathematics and physical sciences
\bibliography{references}   % name your BibTeX data base

\end{document}